%% file: main.tex
\documentclass[submission]{FPSAC2021}

\input{header}

\begin{document}

\maketitle

\input{./sections/intro.tex}


\section{Coxeter groups and geometry: a brief guided tour}\label{sec:prerequisites}
\input{./sections/1-Prerequisites/1-definition.tex}

\subsection{Geometric representations and based root systems}\label{part:roots}
\input{./sections/1-Prerequisites/2-root_systems.tex}

\subsection{Inversion sets}\label{part:inversions}
\input{./sections/1-Prerequisites/3-inversion_sets.tex}

\subsection{The projective picture}\label{part:projective}
\input{./sections/1-Prerequisites/4-projective_picture.tex}

\subsection{The Tits cone}\label{part:tits}
\input{./sections/1-Prerequisites/5-tits_cone.tex}


\section{The Dyer-Hohlweg conjecture}\label{sec:conjecture}

\input{./sections/2-Conjecture/conjecture.tex}


\section{The result}\label{sec:result}

\input{./sections/3-Result/1-strategy.tex}

\subsection{Wiggling hyperplanes}\label{part:wiggling}
\input{./sections/3-Result/2-lemmas.tex}

\subsection{$G_{bip}(w)$ is acyclic.}\label{part:acyclic}
\input{./sections/3-Result/3-acyclic.tex}

\subsection{The sources of $G_{bip}(w)$ are descents.}\label{part:sources}
\input{./sections/3-Result/4-sources.tex}


\section{Final remarks}

It can be noted that the proof only requires the bipodality property of the set of small roots and thus can be extended to other bipodal sets. In particular, \cite[Conjecture 3]{dyer2016small} proposes that the set of $n$-small roots (a generalization of small roots)
is always bipodal. 
If this is the case, the proof we have presented here generalizes to $n$-low elements and $n$-small inversion sets. 
It is our hope that this partial result in rank three could be used as a lemma for higher rank as in \cite[Lemma 4.21]{dyer2016small}, through the careful choice of rank 3 reflection subgroups. 
Finally, considering inversion polytopes by themselves yields another special case: 
one can show that, when all small roots lie on the edges of the simplex formed by $\conv(\Delta)$ in the projective picture, all small inversion polytopes $\conv(\lambda)$, $\lambda \in \Lambda$,  are inversion polytopes.
This happens for instance in the so called \emph{right-angled case} or when no pair of generators in $S$ commute.


\acknowledgements{This work was initiated during an internship under the supervision of Christophe Hohlweg in 2017. The author is grateful to Christophe Hohlweg for tutorship, Nathan Chapelier-Laget for helpful conversations, and Nicolas M. Thiéry for advice on writing.}

\printbibliography

\end{document}

%% file: header.tex
{\theoremstyle{plain}
	\newtheorem{thm}{Theorem}
	\newtheorem{prop}[thm]{Proposition}
	\newtheorem{lem}[thm]{Lemma}
	\newtheorem{conj}{Conjecture}
}

{\theoremstyle{definition}
	\newtheorem*{rmk}{Remark}
	\newtheorem{dftn}[thm]{Definition}
}

\DeclareMathOperator{\conv}{conv}
\DeclareMathOperator{\cone}{cone}
\DeclareMathOperator{\Span}{Span}

\newcommand{\sep}{\:|\:}

\usepackage{lipsum}
\usepackage{todonotes}
\usepackage[utf8]{inputenc}
\usepackage{caption}
\usepackage{subcaption}
\usepackage{hyperref}

\title{Low elements and small inversion sets are in bijection in rank 3 Coxeter groups}

\author{Balthazar Charles\addressmark{1}\thanks{\href{mailto:balthazar.charles@lri.fr}{balthazar.charles@lri.fr}}}

\address{\addressmark{1}Université Paris-Saclay, CNRS, Laboratoire de recherche en informatique, 91405, Orsay, France}

\received{\today}

\abstract{In this extended abstract we announce a proof that, in a Coxeter group of rank 3, low elements are in bijection with small inversion sets. This gives a partial confirmation of Conjecture 2 in [Dyer, Hohlweg '16]. That same article provides the main ingredient: the \emph{bipodality} of the set of small roots is used to propagate information on the vertices of \emph{inversion polytopes}.}

\keywords{Infinite Coxeter group, low element, inversion polytope.}

\usepackage[backend = bibtex, sorting = none]{biblatex}

%% file: sections/intro.tex
One of the reasons that Coxeter groups are so well studied, other than their ubiquity, is perhaps their position at the intersection of group theory, combinatorics and geometry. 
A most impressive example of a result made possible by this confluence is the proof by Brink \& Howlett in \cite{brink1993finiteness} of the automatic structure of 
infinite 
Coxeter groups. 
Indeed, the combinatorial construction of an automaton recognizing the reduced words of the Coxeter group is based on the notion of \emph{small roots} which is geometric in nature. 
In \cite{hohlweg2016automata}, Hohlweg, Nadeau \& Williams explore the connection of \emph{Garside shadow} and \emph{low elements} with families of automatas recognizing the reduced words of a Coxeter group.
These two objects, the latter being an example of the former, were introduced by Dyer \& Hohlweg in \cite{dyer2016small} and were motivated by questions of decidability of the word problem in Braid groups. 

In this paper, we focus our study on low elements and specially on Conjecture~2 of \cite{dyer2016small}, where Dyer \& Hohlweg ask whether low elements are in bijection with \emph{small inversion sets}. This conjecture, if true, would give a connection with  \emph{Shi arrangements} and can in fact be formulated as "low elements are in bijection with Shi regions".
As Shi arrangements and their regions are extensively studied (\cite{fishel2012counting} or \cite{levear2020bijection} for instance), this could reveal a structure on the set of low elements or, conversely, provide a model for the Shi regions.

In this extended abstract, we announce a proof to this conjecture in the case of rank 3 Coxeter groups. In Section~\ref{sec:prerequisites}, we recall notions on Coxeter groups with emphasis on geometric constructions: root systems, inversion sets, projective picture, and the Tits cone. In the following Section~\ref{sec:conjecture} we present the conjecture. Finally, in Section~\ref{sec:result}, we explain our strategy for the proof in rank 3 before sketching the demonstrations of its main steps.

%% file: sections/1-Prerequisites/1-definition.tex
In this section we recall notions related to Coxeter groups necessary both for the statement of the conjecture in Section~\ref{sec:conjecture} and the sketches of proofs in Section~\ref{sec:result}.

%% file: sections/1-Prerequisites/2-root_systems.tex
Let $V$ be a finite dimensional real vector space equipped with a symmetric bilinear form $B$.
A \emph{simple system} in $(V, B)$ is a set $\Delta \subset V$ of $\emph{simple roots}$ such that :
\begin{enumerate}
	\item $\Delta$ is positively linearly independent meaning that no line trough $0$ is contained in the cone $\cone(\Delta) = \sum_{\delta \in \Delta} \mathbb{R}_+\delta$ generated by $\Delta$.
	\item For all $\delta \in \Delta$, $B(\delta, \delta) = 1$.
	\item For all \emph{distinct} $\delta, \delta' \in \Delta$, $B(\delta, \delta') \in\: ]-\infty, -1] \cup \{-\cos(\pi/m_{\delta, \delta'}) \sep m_{\delta, \delta'} \in \mathbb{N}_{\geq 2}\}$.
\end{enumerate}
Let $S = \{s_{\delta} \sep \delta \in \Delta\}$ where $s_{\delta} : x \mapsto x - B(x, \delta)\delta$ is the $B$-reflection associated with $\delta$. 
Denote by $W$ the subgroup of the group of $B$-orthogonal linear maps generated by $S$; we say that $W$ is a \emph{Coxeter group} and that $(W,S)$ is a \emph{Coxeter system} of \emph{rank} $|S|$. 
The \emph{length} of an element $w \in W$, denoted by $|w|$, is the minimal length of a product of generators equal to $w$. For all $(w, s) \in W \times S$, $|ws| = |w| \pm 1$. When $|ws| = |w|-1$, we say that $s$ is a \emph{descent of $w$} and that $ws$ \emph{is covered by} $w$. The transitive reflexive closure of the covering relation is called the \emph{right weak order} and is denoted by $\leq_R$.


We say that $(V, B)$ is a \emph{geometric representation} of $W$. 
When $\Delta$ is linearly independent and $B(e_s, e_t) \geq -1$, we get the $\emph{classical representation}$. 
The orbit of $\Delta$ under $W$ is called the \emph{root system} and is denoted by $\Phi$ (see Figure~\ref{fig:inversion_set}). 
The pair $(\Phi, \Delta)$ is called a \emph{based root system} (as used in \cite{hohlweg2014asymptotical}). 
Based root systems enjoy the following usual properties: 
\begin{itemize}
	\item The root system is partitioned in \emph{positive roots} $\Phi^+ = \Phi\cap\cone(\Delta)$ and \emph{negative roots} $\Phi^-=-\Phi^+$.
	\item For $\rho \in \Phi$, $\mathbb{R}\rho \cap \Phi = \{\rho, -\rho\}$.
\end{itemize}
In addition, based root systems restrict well to reflection subgroups. Indeed, let $A \subset \Phi^+$, if $H_A$ is the subgroup of $W$ generated by $B$-reflections associated to $A$, we set $V_A = \Span(A)$, $\Phi_A=H_A(A)$ and $\Delta_A$ to be the roots generating the extreme rays of $\cone(\Phi_A \cap \Phi^+)$; then $(V_A, B_{|V_A})$ is a geometric representation of $H_A$ of which $(\Phi_A, \Delta_A)$ is a based root system. Beware: it needs not be a classical representation even when the original representation is.

This restriction property allows to test if a reflection subgroup $H_A$ is finite:  it is if and only if $\Span(\Phi_A) \cup Q =\{0\}$ where $Q = \{v \in V \sep B(v,v) = 0\}$ is the \emph{isotropic cone of $B$}.

%% file: sections/1-Prerequisites/3-inversion_sets.tex
For the remainder of this paper, we fix a Coxeter system $(W, S)$ of \emph{finite} rank. 
Many of the combinatorial definitions on Coxeter groups have geometric equivalents through \emph{inversion sets}. We recall below some of their properties (see \cite{humphreys1990reflection} for details).

\begin{prop}\label{prop:inversion_sets}
	Let $w \in W$ be a group element. The (right) \emph{inversion set of $w$} is $N(w) = \Phi^+ \cap w(\Phi^-)$.
	The inversion set map from $(W, \leq_R)$ to the power set $(\mathfrak{P}(\Phi), \subseteq)$ of $\Phi$ is an increasing injection. Moreover, $s$ is a descent of $w$ if and only if $N(w) \setminus \{-w(e_s)\} = N(w')$ for some $w'$.
\end{prop}

\begin{figure}[h] 
	\centering
	\includegraphics[width=0.85\textwidth]{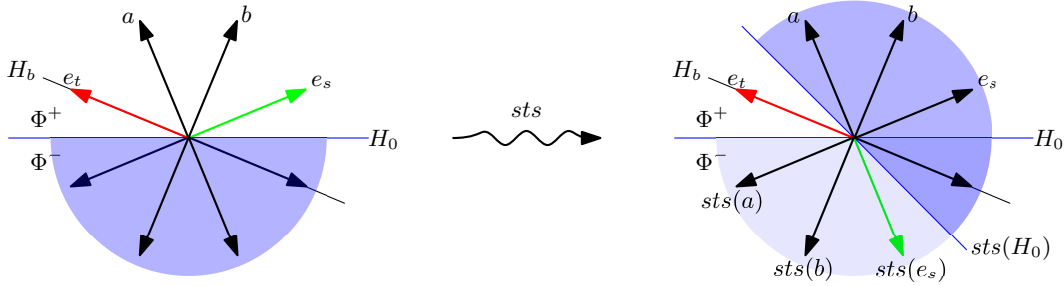}
	\caption{The inversion set of $sts$ (the orthogonal reflection associated with $b$) in the dihedral group $D_4$: $N(sts) = \{e_s, a, b\}$}
	\label{fig:inversion_set}
\end{figure}

Denote by $\Gamma_w = \{\gamma \in N(w) \sep \exists w', N(w) \setminus \{\gamma\} = N(w')\}$ the set of \emph{right geometric descents}, or descents for short. We want to draw the reader's attention on the fact that $\Gamma_w = \emptyset \Leftrightarrow w=1 \Leftrightarrow N(w) = \emptyset$ and that $|\Gamma_w| = |S|$ is possible only if $W$ is finite, in which case $w$ is uniquely defined by $N(w) = \Phi^+$. When this occurs, $w$ is called the \emph{maximal element of $W$}. This will be useful later in Lemma~\ref{lem:sources}.

The question of deciding whether a subset of $\Phi^+$ is an inversion set arises naturally. Fortunately, Proposition~2.11 of \cite{hohlweg2016inversion} gives a simple characterization.
\begin{prop}
	Let $\Phi$ be the root system of a geometrical representation of $(W, S)$. A subset $A$ of $\Phi^+$ is an inversion set if and only if it is finite and \emph{separable}: there exists some hyperplane $H$ such that $A$ is strictly on one side of $H$ and $\Phi^+ \setminus A$ is strictly on the other side.
\end{prop}


%% file: sections/1-Prerequisites/4-projective_picture.tex
From now on, we set $(V, B)$ to be the classical geometric representation of $(W, S)$ and $(\Phi, \Delta)$ the corresponding based root system.

In the case of infinite Coxeter groups, $\Phi$ can prove challenging to picture: there are infinitely many roots, they are of unbounded norm... 
Following \cite{hohlweg2014asymptotical}, we may obtain a handy depiction of the root system: the \emph{projective picture}. 
To a root $\rho$ we associate a \emph{normalized root} $\widehat{\rho} = \mathbb{R}\rho\cap H_1$, where $H_1$ is the affine hyperplane generated by $\Delta$. 
The normalized roots form the set $\widehat{\Phi} = \{\widehat{\rho} \sep \rho\in \Phi\}$, as shown in Figure~\ref{fig:normalisation}. 
We call \emph{projective picture} the set $\widehat{\Phi}$ seen as embedded in $H_1$.
Since $\Phi^+$ is naturally in bijection with $\widehat{\Phi}$, we will identify $a \in \Phi^+$ with $\hat{a}$.

\begin{figure}[h] 
	\centering
	\includegraphics[width=0.9\textwidth]{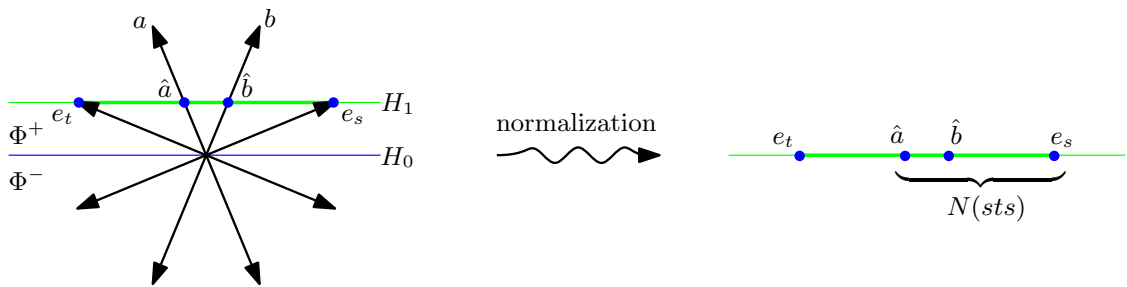}
	\caption{The projective picture for the dihedral group $D_4$.}
	\label{fig:normalisation}
\end{figure}

\begin{rmk}
	Since $\Phi^+ \subset \cone(\Delta)$ and $\Delta \subset H_1$, we have $\widehat{\Phi} \subset \conv(\Delta)$, the convex hull of $\Delta$. Because of this, the projective picture is essentially the image of $\Phi$ by the canonical projection in the projective space $\mathbb{P}(V)$. 
\end{rmk}

More than a simple tool to draw pictures of rank 3 or 4 root systems, the projective picture has several advantages. Because it is a subset of a compact set, $\widehat{\Phi}$ has accumulation points which, by \cite{hohlweg2014asymptotical}, are on the isotropic cone (strictly speaking, on $Q_1 = Q \cap H_1$). 
This circumscribes the treatment of many topological issues to a neighborhood of $Q_1$.
In our proofs in Section~\ref{sec:result}, it will also allow us to simplify the formalism: cones become convex hulls, extreme rays become points, etc. This allows us to translate the inversion set into an "even more" geometric object.

\begin{dftn}
	The \emph{inversion polytope} of an element $w$ in $W$ is the convex hull of $N(w)$ in $H_1$. It is denoted by $\mathcal{P}_w$. The set of its vertices is denoted by $N^1(w)$.
\end{dftn}

Recall that $N(w)$ is separable. This too translates in the projective picture: there exists some affine hyperplane $H$ of $H_1$ strictly separating the inversion polytope $\mathcal{P}_w$ from its complement in $\widehat{\Phi}$. This imposes that $N(w)$ is convex in the sense that $\mathcal{P}_w \cap \widehat{\Phi} = N(w)$ (see \cite[Lemma 2.10]{hohlweg2016inversion}): $\mathcal{P}_w$ and $N(w)$ hold the same information.

%% file: sections/1-Prerequisites/5-tits_cone.tex
We present here the notion of \emph{Tits cone} which can be understood as a dual of the root system. This dual point of view will be useful in stating simply the proofs in  \S\ref{part:wiggling} where we examine the interactions of inversion sets and moving separation hyperplanes. We refer to \cite{abramenko2008buildings} for more details.

Recall that we have fixed $(W, S)$ of finite rank, which means that $V$ is finite dimensional.
We denote by $\langle\,\cdot\sep\cdot\, \rangle: V \times V^* \longrightarrow \mathbb{R}$ the duality bracket. 
For $v \in V$ we set $H_v = \{f \in V^* \sep \langle v\sep f\rangle = 0\}$, $H^+_v = \{f \in V^* \sep \langle v\sep f\rangle > 0\}$ and $H^-_v = - H^+_v$. 
The map $v \mapsto (H_v \cup H_v^+)$ is a bijection from the sphere of $V$ to the set of closed half-spaces of $V^*$ which is bicontinuous.

Let the set $\mathcal{C} = \bigcap_{\delta \in \Delta} H^+_{\delta} = \bigcap_{\rho \in \Phi}H_{\rho}^+$ be the \emph{fundamental chamber} and $\overline{\mathcal{C}}$ its closure for the usual topology,
then the \emph{Tits cone of $W$} is $U = \bigcup_{w \in W}w(\overline{\mathcal{C}})$. 
The connected components of $U \setminus \bigcup_{\rho \in \Phi}H_{\rho}$ are called the \emph{Weyl chambers}. 
They are in bijection with $W$ via $w \longmapsto w(\mathcal{C})$, where $w$ acts on $\mathcal{C}$ by duality. 
Fix a root $\rho \in \Phi$ and a Weyl chamber $w(\mathcal{C})$. We say that $H_{\rho}$ is a \emph{wall of $w(\mathcal{C})$} if $H_{\rho} \cap w(\overline{\mathcal{C}})$ spans $H_{\rho}$. Walls give us yet another way to recognize inversions and descents.

\begin{prop}\label{prop:tits}
	For any $w \in W$, $\rho \in N(w)$ if and only if $w(\mathcal{C}) \subset H_{\rho}^-$. Moreover $\rho$ is a descent of $w$ if and only if $\rho \in N(w)$ and $H_{\rho}$ is a wall of $w(\mathcal{C})$.
\end{prop}

\begin{figure}[h]
	\centering
	\includegraphics[width=\textwidth]{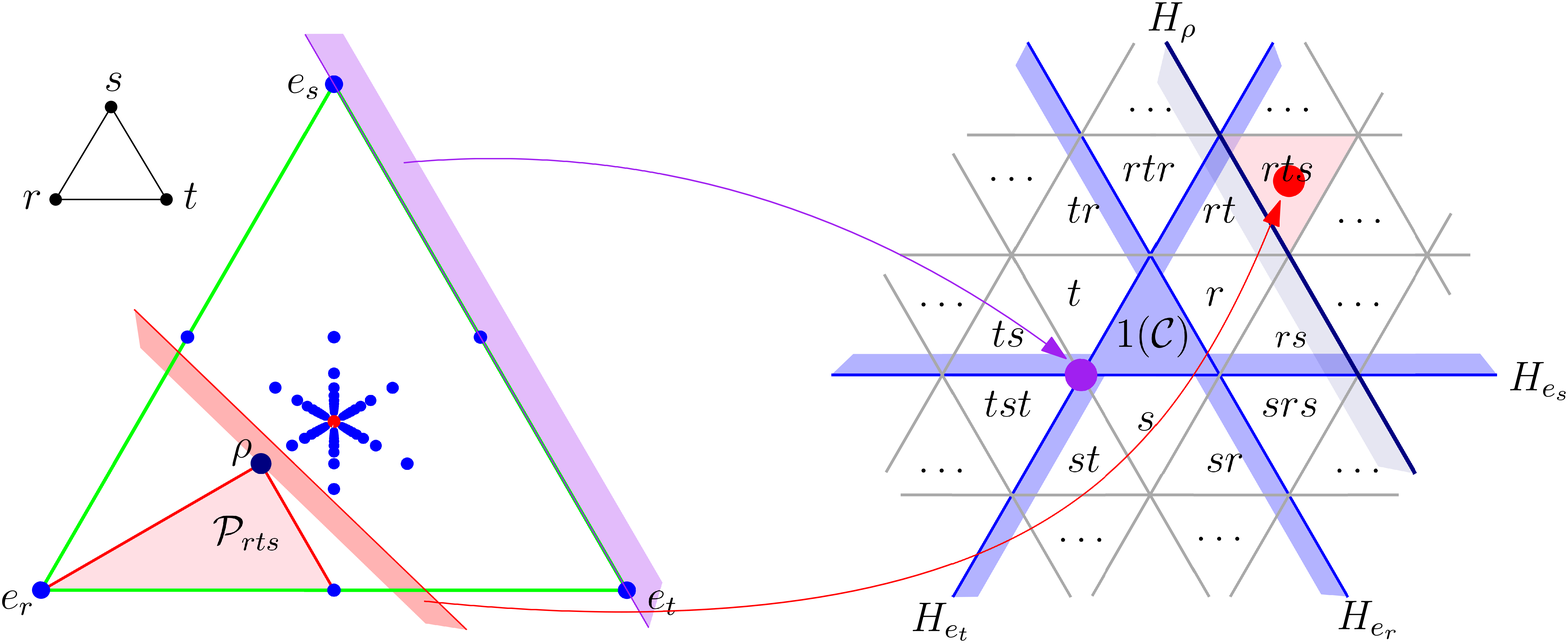}
	\caption{On the left, the projective picture. On the right, the \emph{Coxeter complex} $(U \setminus\{0\}) / \mathbb{R}^+$. Weyl chambers are labeled by their associated element. Half-spaces in the primal are sent to points of the same color in the Coxeter complex. Conversely, the descent $\rho$ of $rts$ is sent to the wall $H_{\rho}$ of $rts(\mathcal{C})$ with $rts(\mathcal{C}) \subset H_{\rho}^-$.}
	\label{fig:duality}
\end{figure}

%% file: sections/2-Conjecture/conjecture.tex
In this section, we state the conjecture, briefly explain its connection with Shi regions, and present the notion of bipodality.

\begin{dftn}
	We say that $a \in \Phi^+$ is a \emph{small root} if for all $b \in \Phi^+\setminus\{a\}$, there is some $w \in W$ such that $a \in N(w)$ and $b \notin N(w)$. The set of small roots is denoted by $\Sigma$. 
	
	For $w \in W$, $\Sigma(w) = \Sigma \cap N(w)$ is the \emph{small inversion set of $w$}. We denote by $\Lambda$ the set of all small inversion sets.
	
	An element $w$ of $W$ is a \emph{low element} if $N(w) = \cone(\Sigma(w)) \cap \Phi^+$. The set of low elements is denoted by $L$.
\end{dftn}

Said in the language of the projective picture, where extreme rays of cones become extreme points of convex sets, an element $w \in W$ is low if and only if the vertices of its inversion polytope are small roots; that is $N^1(w) \subset \Sigma$.
Because we already know that the map $w \longmapsto N(w)$ is injective, the map $\Sigma: w \in L \longmapsto \Sigma(w) \in \Lambda$ is injective as well. Is it surjective? 

\begin{conj}\cite[Conjecture 2, Dyer, Hohlweg '16]{dyer2016small}\label{conj:bij}
	The map $\Sigma: L \longrightarrow \Lambda$ is a bijection between low elements and small inversion sets.
\end{conj}

Let us restate this in terms of \emph{Shi regions}, that is, the connected components of $U \setminus \mathcal{A}_{\Sigma}$ where $U$ is the Tits cone and $\mathcal{A}_{\Sigma} = \{H_{a} \sep a \in \Sigma\}$ is the \emph{Shi arrangement} (this is equivalent to the usual definition, see \cite[p.123]{bjorner2006combinatorics}). 
As the Weyl chambers are in bijection with inversion sets, the Shi regions are naturally in bijection with the small inversion sets. Considering this, the question becomes "are low elements in bijection with Shi regions?".

In the same paper where Dyer \& Hohlweg introduce the conjecture (\cite{dyer2016small}), they define the notion of \emph{bipodality} and prove that the set of small roots is bipodal. This property will be central in the proof in the next section. We illustrate it in Figure~\ref{fig:bipodality}.

\begin{figure}[h] 
	\centering
	\includegraphics[width=0.45\textwidth]{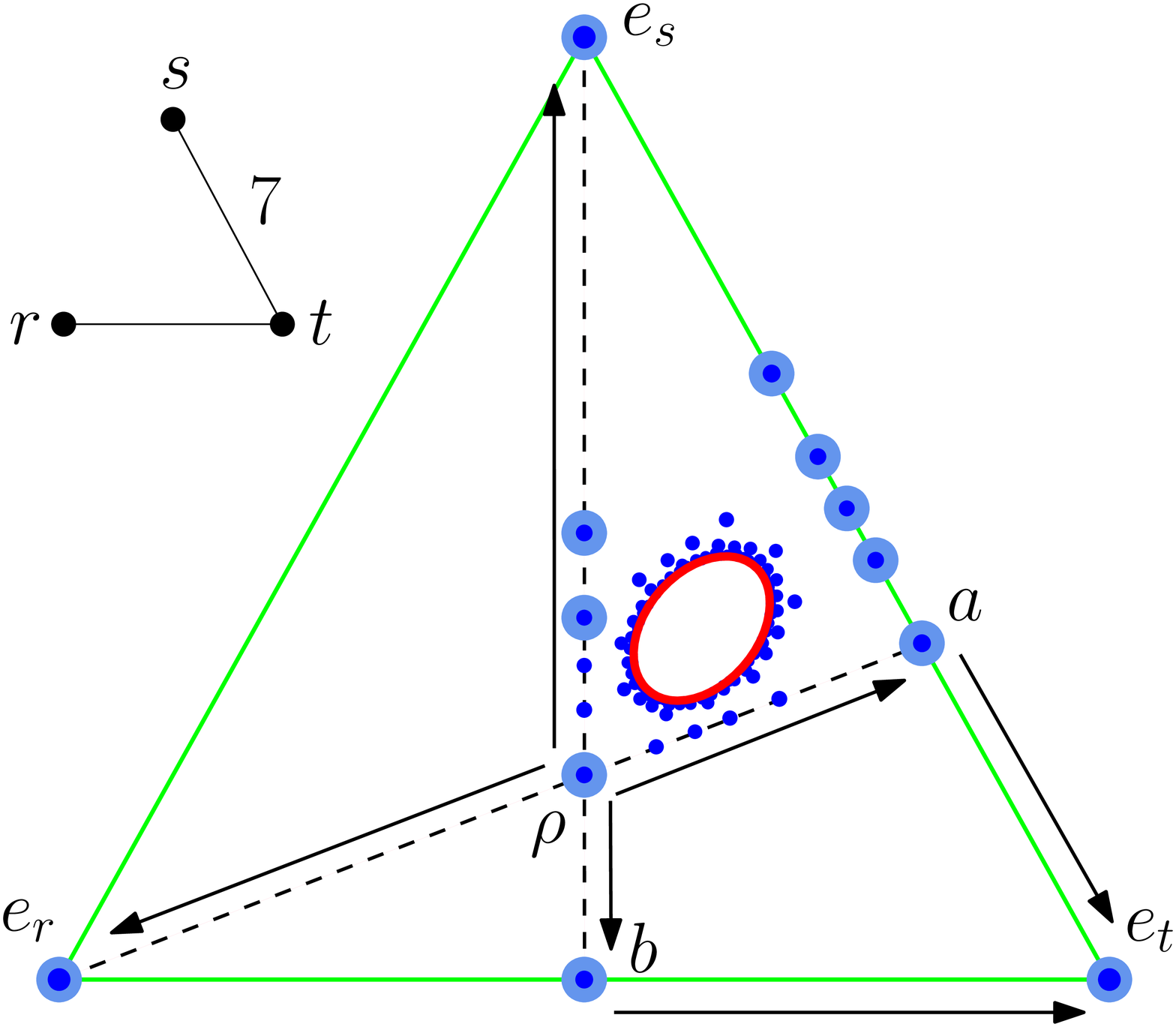}
	\caption{Bipodality of the small roots (with the blue halo): because the small root $\rho$ is an "internal" root of the dashed lines, it imposes that the "extremal" roots on these lines (here $a$ and $e_r$, $b$ and $e_s$) are small too. In turn, this forces $e_t$ to be small.}
	\label{fig:bipodality}
\end{figure}

\begin{dftn}\label{def:bipodality}
	Let $d$ be a line in the projective picture and, as in \S\ref{part:roots}, let $H_d$ be the reflection subgroup corresponding to $\Phi^+_d$, the roots contained in $d$.
	We say that the pair $(a, b)$ is an \emph{arrow} from $a$ to $b$
	if $b \in \Delta_d$ is a simple root 
	and $a \in \Phi_d^+ \setminus \Delta_d$ is not, 
	and denote it $a \rightarrow b$.
	A subset $A \subset \Phi^+$ is said to be \emph{bipodal} if for any arrow $a \rightarrow b$, $a \in A$ implies $b \in A$.
\end{dftn}

\begin{thm}\cite[Dyer, Hohlweg '16, Theorem 4.18]{dyer2016small}
	$\Sigma$ is always bipodal.
\end{thm}

%% file: sections/3-Result/1-strategy.tex
We are now in position to state our main result -- Conjecture~\ref{conj:bij} holds in rank 3 -- and sketch its proof.
\begin{prop}
	In rank 3, $\Sigma: L \longrightarrow \Lambda$ is surjective.
\end{prop}

\begin{proof}[Strategy]
	Recall that an element $w$ is low if and only if the vertices of $\mathcal{P}_w$ are in $\Sigma$. Thus, we can prove this result by exhibiting for any small inversion set $\lambda \in \Lambda$ an element $w$ such that $N^1(w) \subset \Sigma$ and $\Sigma(w) = \lambda$. Our strategy is as follows:
	\begin{enumerate}
		\item Choose $w$ with $\Sigma(w) = \lambda$ such that its descents are small, \emph{.ie.} $\Gamma_w \subset \Sigma$ (Lemma~\ref{lem:reduction}).
		\item Define the \emph{bipodality graph} $G_{bip}(w)$ as the directed graph whose vertices are those of $\mathcal{P}_w$ and whose edges are those of $\mathcal{P}_w$ that are arrows in the sense of definition~\ref{def:bipodality}.
		\item Prove that every vertex on the bipodality graph is accessible from $\Gamma_w$: show that $G_{bip}(w)$ is acyclic (Proposition~\ref{prop:acyclic}) and that its sources lie in $\Gamma_w$ (Proposition~\ref{prop:sources}). By bipodality, this imposes that $N^1(w) \subset \Sigma$, proving that $w$ is a low element. \qedhere
	\end{enumerate}
\end{proof}

The most difficult point is to show that the sources form a subset of the descents. We present here a proof in rank 3 and it is the only part of the proof that does not easily extend to higher ranks. We begin by some remarks on the interaction between moving separations hyperplane and inversion polytopes in \S\ref{part:wiggling}. This gives us some tools to prove that $G_{bip}(w)$ is acyclic in \S\ref{part:acyclic} and to localize the sources of the graph in \S\ref{part:sources}.

%% file: sections/3-Result/2-lemmas.tex
As can be seen on Figure~\ref{fig:duality}, the set of separation hyperplanes $H$ for $N(w)$ is homeomorphic to $w(\mathcal{C}) / \mathbb{R}^+$. 
This shows, for instance, that if $H, H'$ are two separation hyperplanes for $N(w)$, we can continuously move one to the other.
In this spirit, we examine in this section some interactions between roots and moving hyperplanes, beginning with a relaxation of the separability condition.

\begin{dftn} \label{def:weak_separation}
	Let $w$ be an element of $W$ and $\mathcal{P}_w$ its inversion polytope. We say that an affine hyperplane $H \subset H_1$ is a \emph{weak separation hyperplane} for $\mathcal{P}_w$ if $\mathcal{P}_w$ is on one side of $H$, while its complement $\widehat{\Phi} \setminus \mathcal{P}_w$ is strictly on the other side and $H$ does not intersect the normalized isotropic cone.
\end{dftn}

\begin{figure}
	\centering
	\begin{subfigure}[b]{0.3\textwidth}
		\centering
		\includegraphics[width=\textwidth]{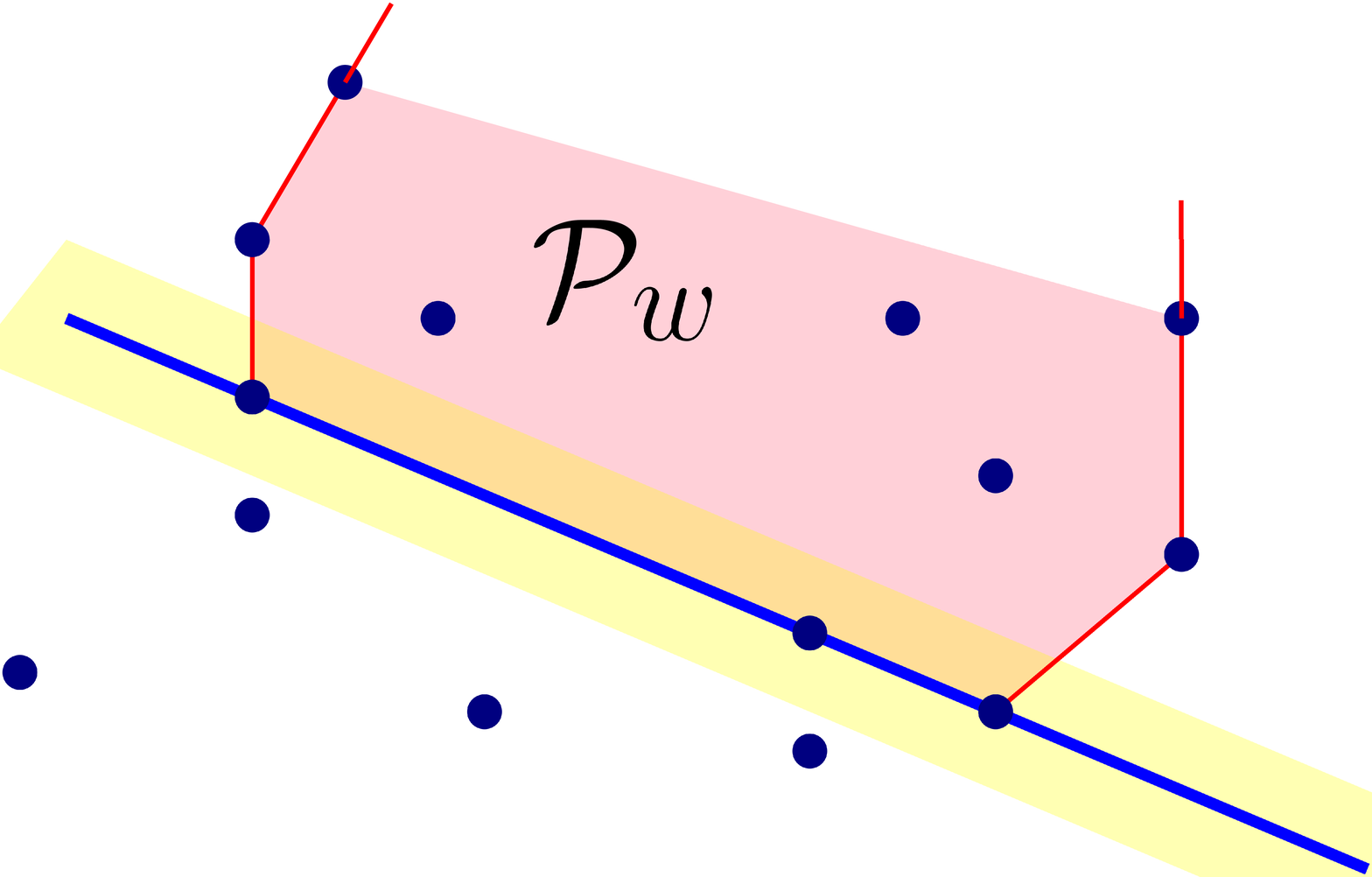}
		\caption{Weak separation...}
		\label{fig:weak_separation}
	\end{subfigure}
	\hfill
	\begin{subfigure}[b]{0.3\textwidth}
		\centering
		\includegraphics[width=\textwidth]{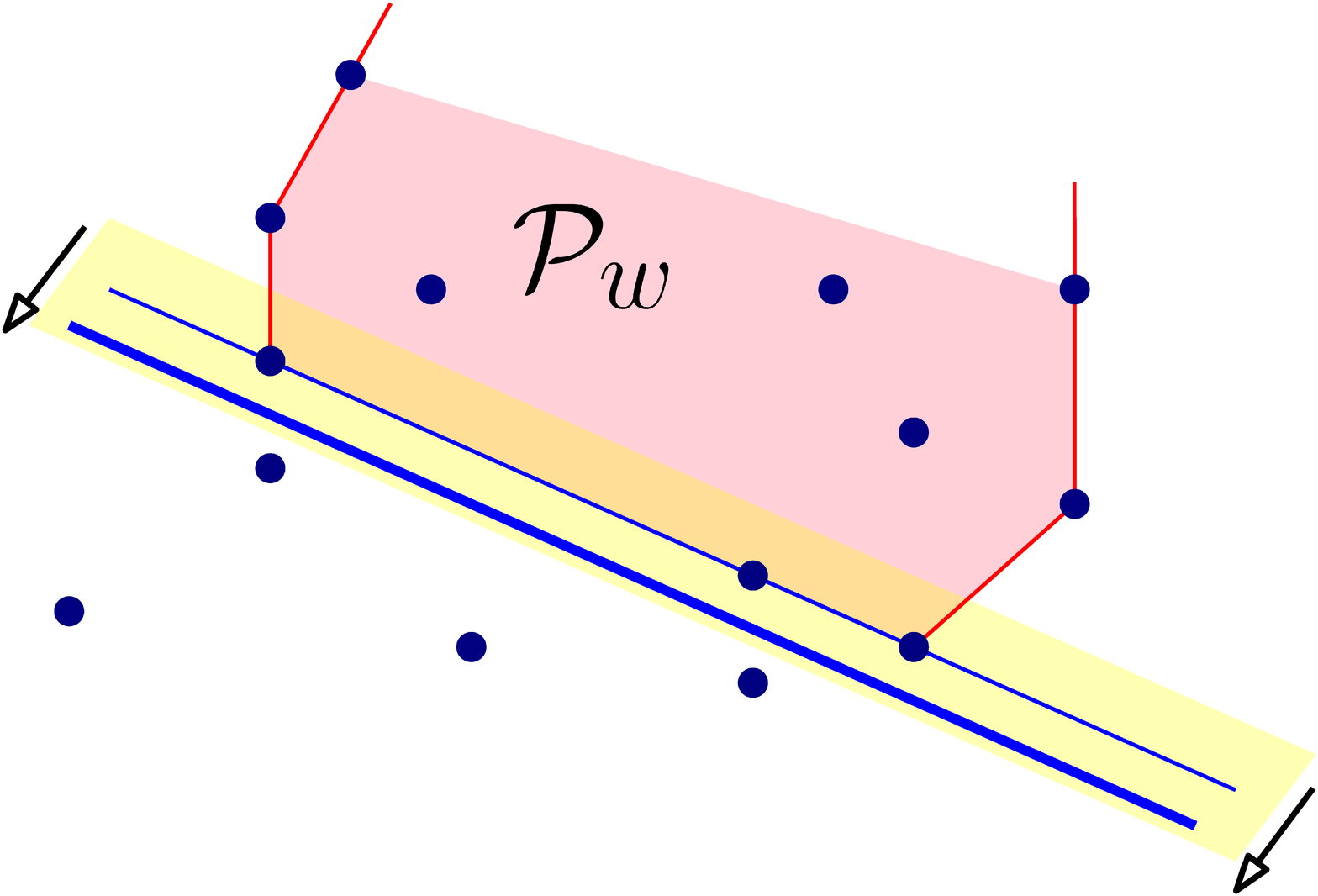}
		\caption{...to regular...}
		\label{fig:strict_separation}
	\end{subfigure}
	\hfill
	\begin{subfigure}[b]{0.3\textwidth}
		\centering
		\includegraphics[width=\textwidth]{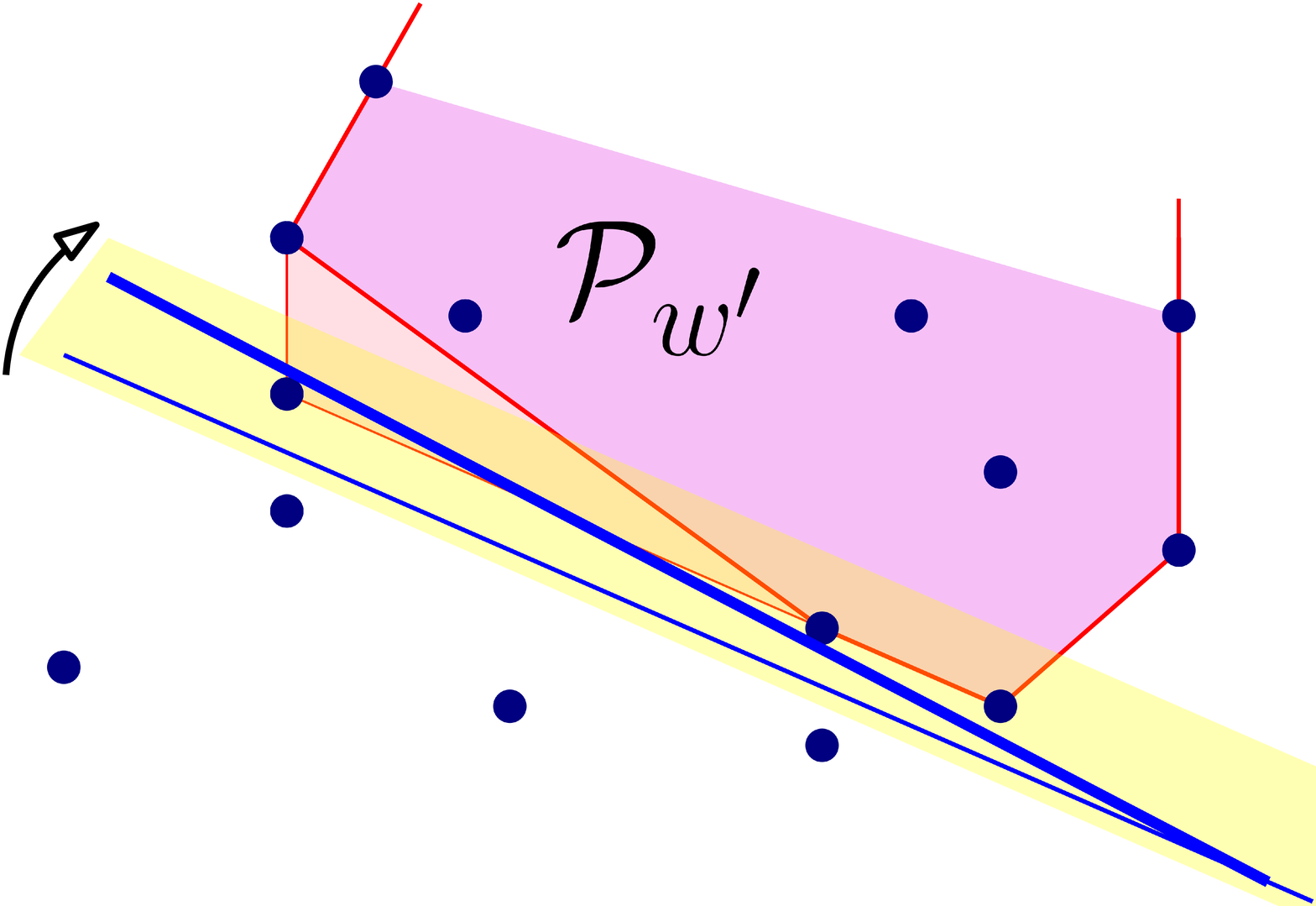}
		\caption{...to removing a descent.}
		\label{fig:tilt_descentes}
	\end{subfigure}
	
	\hspace{2mm}
	
	\begin{subfigure}[b]{0.25\textwidth}
		\centering
		\includegraphics[width=\textwidth]{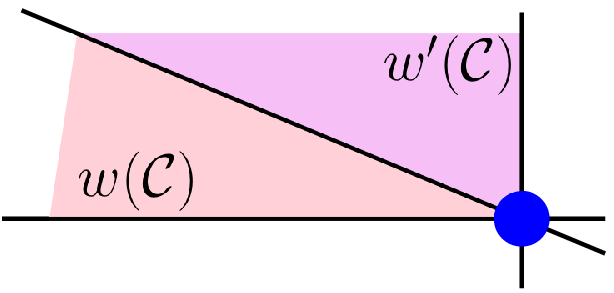}
	\end{subfigure}
	\hfill
	\begin{subfigure}[b]{0.25\textwidth}
		\centering
		\includegraphics[width=\textwidth]{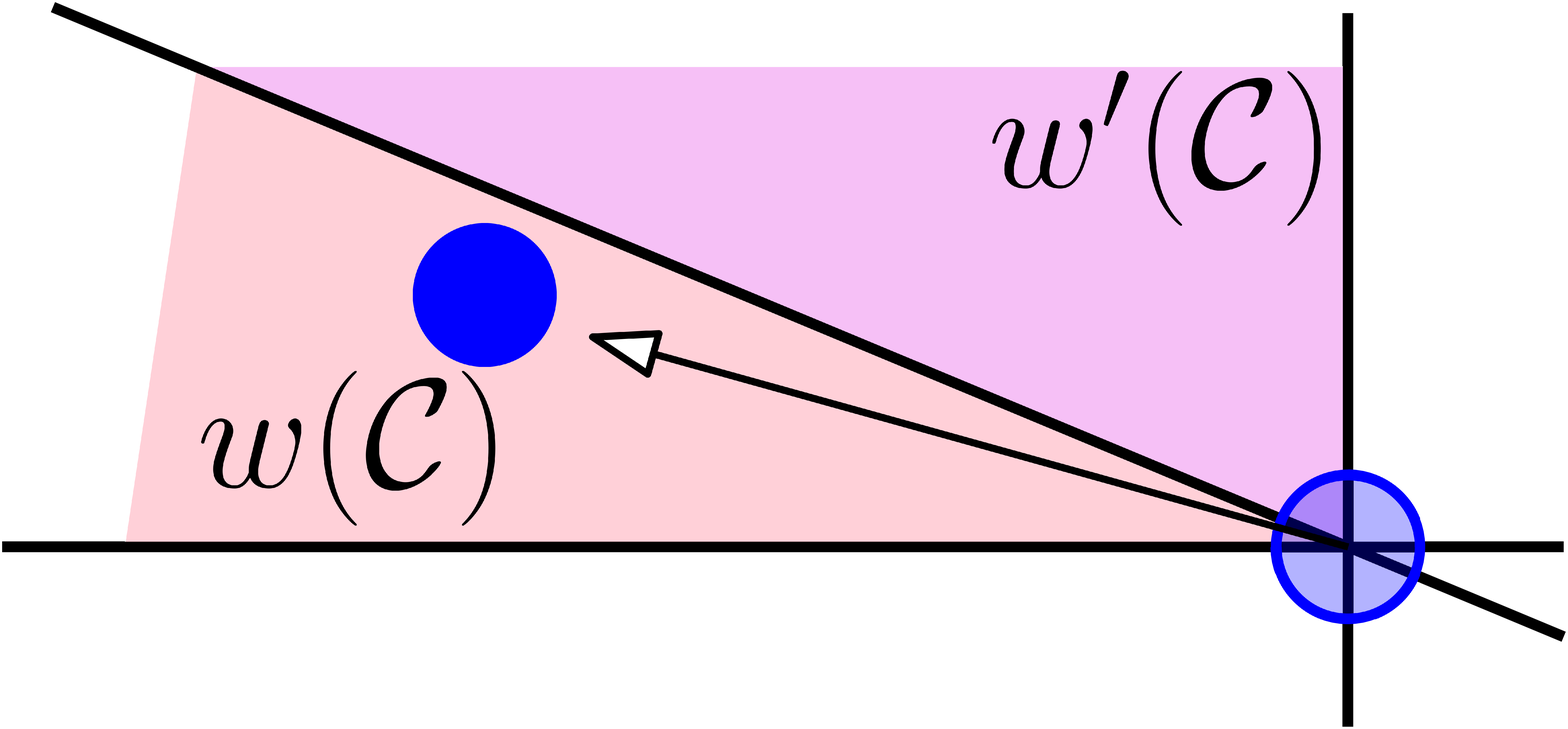}
	\end{subfigure}
	\hfill
	\begin{subfigure}[b]{0.25\textwidth}
		\centering
		\includegraphics[width=\textwidth]{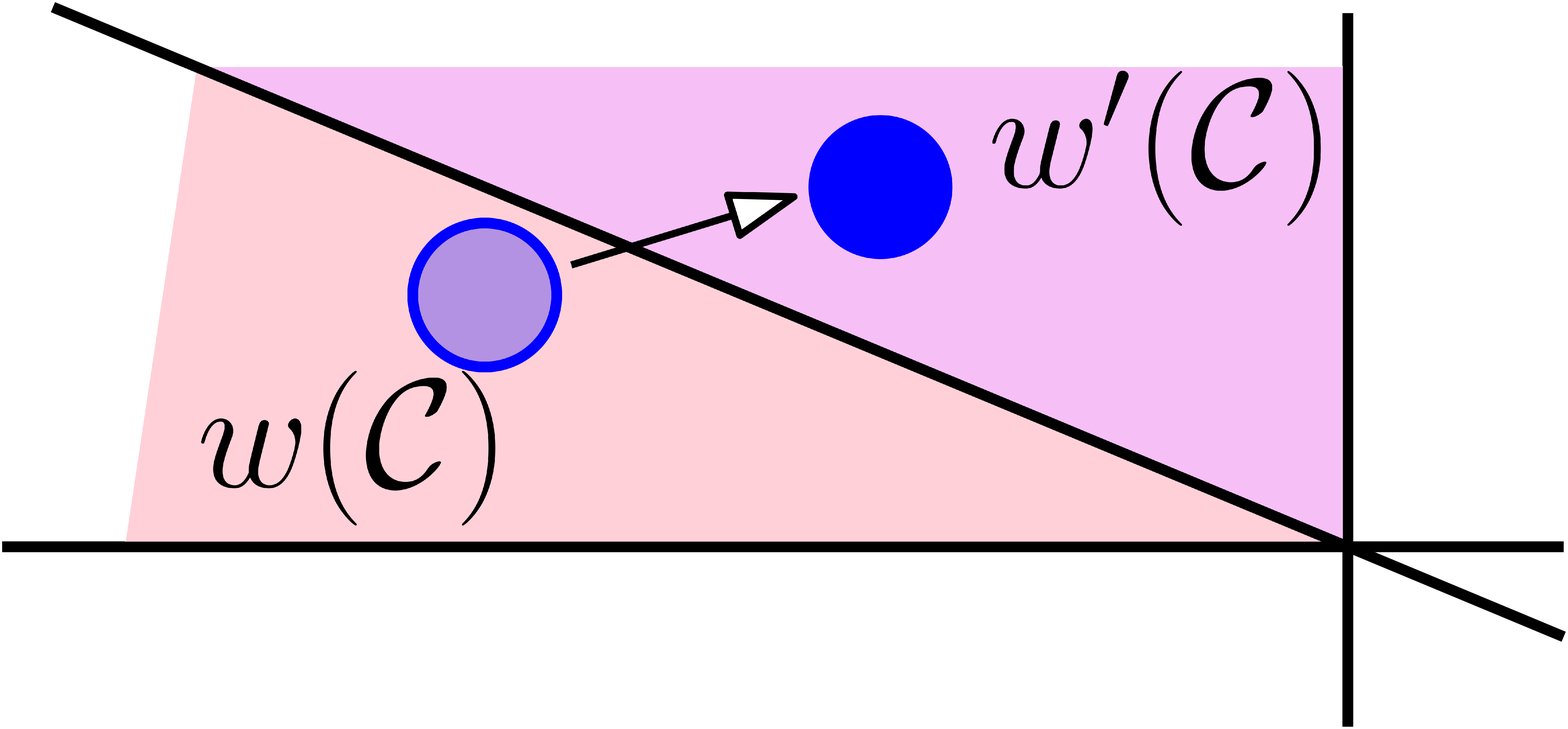}
	\end{subfigure}
	\quad
	\caption{Wiggling hyperplanes : movements of the separation hyperplane on the top row (primal) correspond to a movement of the associated point of the Coxeter complex in the bottom row (dual). The root-free envelope is figured in yellow.}
	\label{fig:weaker_is_better}
\end{figure}

For finite subsets of $\widehat{\Phi}$, this is equivalent to the regular separability condition: a weak separation hyperplane $H$ does not cross the isotropic cone so it is at a strictly positive distance of the roots it does not contain. This gives us a root free envelope (fig. \ref{fig:weak_separation}) in which we may translate $H$ to make it into a strict separation hyperplane (fig. \ref{fig:strict_separation}). Because the set $\widehat{\Phi}$ is bounded, we may actually exit this envelope, as long as we do it "far enough": we can tilt $H$ to remove the vertices of $\conv(H \cap \widehat{\Phi})$ (fig. \ref{fig:tilt_descentes}). This gives us indications on the descents of an element:

\begin{lem}\label{lem:weaker_is_better}
	Let $w$ be an element of $W$, and $\mathcal{P}_w$ its inversion polytope. We suppose that $H$ is a weak separation hyperplane for $\mathcal{P}_w$ containing a face $F$ of $\mathcal{P}_w$. Then the vertices of $F$ are in $\Gamma_{w}$.
\end{lem}

This lemma expresses a more general idea: because small perturbations of a point in the Coxeter complex correspond to small perturbations of a separation hyperplane, \emph{we can choose to enter a chamber through one of its walls}.

%% file: sections/3-Result/3-acyclic.tex
Let $\lambda \in \Lambda$ be a small inversion set, let us show that we can choose $w \in W$ with $\Sigma(w) = \lambda$ and $\Gamma_w \subset \Sigma$, as per the first point of our strategy.
Indeed, if it is not the case then remove $\gamma \in \Gamma_{w} \setminus \Sigma$. 
Notice that the resulting set of roots is still the inversion set of some $w'$ and that the small inversion set is unchanged. 
Repeat. 
Because the initial $w$ is of finite length, this eventually terminates. We obtained the following lemma:

\begin{lem}\label{lem:reduction} 
	Let $\lambda \in \Lambda$ be a small inversion set. Then there exists some $w \in W$ such that $\Sigma(w) = \lambda$ and $\Gamma_w \subset \Sigma$.
\end{lem}

We now want to show that every vertex of $\mathcal{P}_w$ is accessible from $\Gamma_w$.
It is enough to prove that $G_{bip}(w)$ is acyclic and that the sources of $G_{bip}(w)$ (meaning the vertices connected only to outward edges) are a subset of $\Gamma_w$.
Indeed, to find a path from $\Gamma_w$ to a vertex $v$ we just have to follow the following process: if $v$ is a source, stop. 
Else, $v$ must be connected by an inward edge to $v'$, replace $v$ by $v'$ and iterate. 
Since there is no cycle, this terminates on a source, which is in $\Gamma_w$.

We now fix the $w$ obtained from Lemma~\ref{lem:reduction} for the remainder of this paper. Our next step is to show that the bipodality graph is acyclic. Let us first prove the following useful lemma.

\begin{lem}\label{lem:removal_order}
	We say that a permutation $R = (\rho_1, \rho_2,\dots, \rho_{|w|})$ of $N(w)$ is a \emph{removal order} if for all $i \in \{1,\dots, |w|\}$, $N(w) \setminus \{\rho_j \sep i \leq j \} = N(w_i)$ for some $w_i$. Let $[a, b]$ be an edge of $\mathcal{P}_w$ with $a \rightarrow b$. Then the roots on $[a, b]$ are removed from $a$ to $b$.
\end{lem}

\begin{proof}
	Notations are defined on Figure~\ref{fig:removal}. Suppose we remove $b$ first. Since the remaining roots form an inversion set, it means that a separation hyperplane $H$ must cut the line $(a, b)$ between $a$ and $b$ \emph{and} between $\alpha$ and $a$: this is absurd.
\end{proof}

\begin{figure}[h] 
	\centering
	\includegraphics[width=0.8\textwidth]{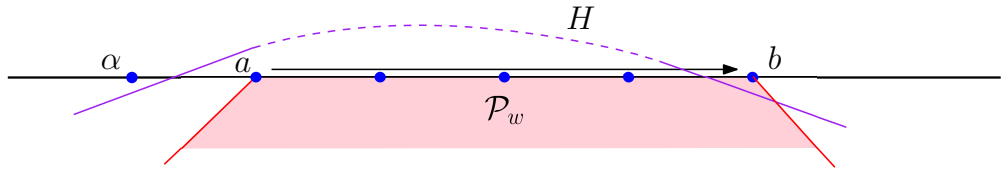}
	\caption{A removal order follows the arrows. $\alpha$ exist by definition of $[a, b]$ being an arrow from $a$ to $b$.}
	\label{fig:removal}
\end{figure}

\begin{rmk}
	Note that this argument has two consequences. Firstly, an edge $[a, b]$ of $\mathcal{P}_w$ is either an arrow or \emph{complete}, meaning $(a, b) \cap \widehat{\Phi} = [a, b]\cap\widehat{\Phi}$. 
	Secondly, it means that any removal order determines an orientation on the edges of $\mathcal{P}_w$. 
	Indeed, even if $[a, b]$ is complete, if $a$ is removed first, the remaining roots on $[a, b]$ must be removed from $a$ to $b$. 
	This allows us to define a (directed) \emph{removal graph} $G_R(w)$ formed of the vertices and edges of $\mathcal{P}_w$ oriented by $R$: if $[a, b]$ is an edge of $\mathcal{P}_w$, we note $a \rightarrow_R b$ the corresponding edge in $G_R(w)$.
\end{rmk}

We can now restate Lemma~\ref{lem:removal_order} as: for any order $R$, $G_{bip}(w)$ is a subgraph of $G_R(w)$. This will help us to reach our goal for this paragraph:

\begin{prop}\label{prop:acyclic}
	The bipodality graph is acyclic.
\end{prop}

\begin{proof}
	Because $G_{bip}(w)$ is a subgraph of $G_R(w)$ for any removal order $R$, it is enough to show that $G_R(w)$ is acyclic for some $R$. Let $H$ be a separation hyperplane. In the spirit of Lemma~\ref{lem:weaker_is_better}, we further assume that $H$ is generic: it is not parallel to any line through two roots in $\mathcal{P}_w$. We continuously translate $H$ through $\mathcal{P}_w$, following its orthogonal line. Because $H$ is generic, the roots will traverse $H$ one at a time: except at these times, $H$ is a separation hyperplane for the remaining roots. Hence, this defines a removal order. Since $a \rightarrow_R b$ is only possible if $H$ visits $a$ before $b$, the related removal graph $G_R(w)$ is acyclic.
\end{proof}

%% file: sections/3-Result/4-sources.tex
The goal of this section is to prove the following proposition:
\begin{prop}\label{prop:sources}
	Assume that $|S| = 3$. Then every source of the bipodality graph on $\mathcal{P}_w$ is a descent.
\end{prop}  

Let $v$ be a source: either it is only connected to outward arrows (case 1), or it is connected to at least one complete edge (case 2). Case 1 is easy: such a source in $G_{bip}(w)$ must be a source in any removal graph. In particular, removal graphs obtained by pushing hyperplanes through polytopes (as in Proposition \ref{prop:acyclic}) only have one source, which is the first root they meet and thus must be a descent.

Case 2 is harder. First, notice that a simple root cannot be a source because it always can be removed last, which means it is a sink in some removal graph. The only remaining possibility is dealt with in the following lemma.

\begin{lem}\label{lem:sources}
	We suppose that $(W, S)$ is a Coxeter system of rank 3. Let $\mathcal{P}_w$ be the inversion polygon of an element $w \in W$. Let $[a, b]$ be a complete edge of $\mathcal{P}_w$. If $a\notin \Delta$, then $a\in \Gamma_w$. 
\end{lem}

\begin{proof}
	Notice first that because $[a, b]$ is complete, $d = (a, b)$ does not intersect the isotropic cone. 
	If we can prove that $d$ is a weak separation line for $\mathcal{P}_w$, from Lemma~\ref{lem:weaker_is_better} we get that $a$ must be a descent. 
	Let $\delta$ be a separation line for $\mathcal{P}_w$, and $d_+$ (resp. $\delta_+$) the closed half-plane associated to $d$ (resp. $\delta$) containing $\mathcal{P}_w$.
	We treat here only the case where $d_+ \cap \widehat{\Phi}$ is finite, that is, is the inversion set of some $w'$.
	If we can move $\delta$ to $d$ without meeting any root, we will have shown that $w = w'$.
	Using Lemma~\ref{lem:weaker_is_better} we define a separation hyperplane $d'$ for $\mathcal{P}_{w'}$.
	We then rotate $\delta$ to $d'$ around their intersection point, in the direction where it does not pass through $\mathcal{P}_w$.
	If we meet a root during the movement, it must be in $J = d_+ \cap (\mathbb{R}^2 \setminus \delta_+)$ (see Figure~\ref{fig:rotation}).
	Up to small perturbations of our movement, we may assume that we meet a unique last root $c$, meaning we enter $w'(\mathcal{C})$ through the wall associated to $c$.
	We have $c \in N(w')$ and by reversing the movement we can remove it and still be left with an inversion set, so $c \in \Gamma_{w'}$.
	Since $a, b \in \Gamma_{w'}$, this means $|\Gamma_w| = 3$ so $w'$ must be the maximal element and $N(w') = \Phi^+$.
	But there is a simple root on either side of $d'$ because $a$ is not a simple root.
	This is absurd.
	The case where $d_+ \cap \widehat{\Phi}$ is infinite is treated similarly by considering the other side of $d$ and interpreting walls as three different elements covering a non neutral element.
\end{proof}

\begin{figure}
	\centering
	\includegraphics[width=0.45\textwidth]{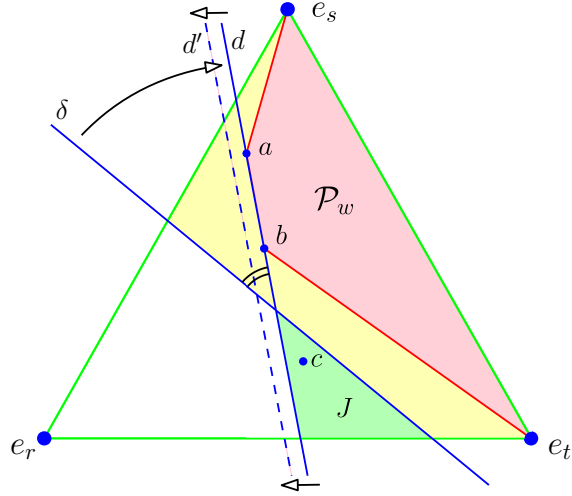}
	\caption{Inversion polygones of $sw$ and $w$ with $w = rtrsrtr$.}
	\label{fig:rotation}
\end{figure}